\documentclass{amsart}
\usepackage{amssymb}
\usepackage{amsfonts}

\newtheorem{theorem}{Theorem}
\theoremstyle{plain}

\newtheorem{corollary}{Corollary}

\newtheorem{lemma}{Lemma}

\newtheorem{proposition}{Proposition}
\newtheorem{remark}{Remark}

\numberwithin{equation}{section}
\input{tcilatex}

\begin{document}
\title[Quadratic Weighted Geometric Mean]{Quadratic Weighted Geometric Mean
in Hermitian Unital Banach $\ast $-Algebras}
\author{S. S. Dragomir$^{1,2}$}
\address{$^{1}$Mathematics, College of Engineering \& Science\\
Victoria University, PO Box 14428\\
Melbourne City, MC 8001, Australia.}
\email{sever.dragomir@vu.edu.au}
\urladdr{http://rgmia.org/dragomir}
\address{$^{2}$DST-NRF Centre of Excellence \\
in the Mathematical and Statistical Sciences, School of Computer Science \&
Applied Mathematics, University of the Witwatersrand, Private Bag 3,
Johannesburg 2050, South Africa}
\subjclass{47A63, 47A30, 15A60, 26D15, 26D10}
\keywords{Weighted geometric mean, Weighted harmonic mean, Young's
inequality, Operator modulus, Arithmetic mean-geometric mean-harmonic mean
inequality.}

\begin{abstract}
In this paper we introduce the \textit{quadratic weighted geometric mean} 
\begin{equation*}
x\circledS _{\nu }y:=\left\vert \left\vert yx^{-1}\right\vert ^{\nu
}x\right\vert ^{2}
\end{equation*}%
for invertible elements $x,$ $y$ in a Hermitian unital Banach $\ast $%
-algebra and real number $\nu $. We show that 
\begin{equation*}
x\circledS _{\nu }y=\left\vert x\right\vert ^{2}\sharp _{\nu }\left\vert
y\right\vert ^{2},
\end{equation*}%
where $\sharp _{\nu }$ is the usual geometric mean and provide some
inequalities for this mean under various assumptions for the elements
involved.
\end{abstract}

\maketitle

\section{Introduction}

Let $A$ be a unital Banach $\ast $-algebra with unit $1$. An element $a\in A$
is called \textit{selfadjoint} if $a^{\ast }=a.$ $A$ is called \textit{%
Hermitian} if every selfadjoint element $a$ in $A$ has real \textit{spectrum}
$\sigma \left( a\right) ,$ namely $\sigma \left( a\right) \subset \mathbb{R}$%
.

In what follows we assume that $A$ is a Hermitian unital Banach $\ast $%
-algebra.

We say that an element $a$ is \textit{nonnegative} and write this as $a\geq
0 $ if $a^{\ast }=a$ and $\sigma \left( a\right) \subset \left[ 0,\infty
\right) .$ We say that $a$ is \textit{positive }and write $a>0$ if $a\geq 0$
and $0\notin \sigma \left( a\right) .$ Thus $a>0$ implies that its inverse $%
a^{-1}$ exists. Denote the set of all invertible elements of $A$ by $%
\limfunc{Inv}\left( A\right) .$ If $a,b\in \limfunc{Inv}\left( A\right) ,$
then $ab\in \limfunc{Inv}\left( A\right) $ and $\left( ab\right)
^{-1}=b^{-1}a^{-1}.$ Also, saying that $a\geq b$ means that $a-b\geq 0$ and,
similarly $a>b$ means that $a-b>0.$

The \textit{Shirali-Ford theorem} asserts that \cite{SF} (see also \cite[%
Theorem 41.5]{BD})%
\begin{equation}
a^{\ast }a\geq 0\text{ for every }a\in A.  \tag{SF}  \label{SF}
\end{equation}%
Based on this fact, Okayasu \cite{O}, Tanahashi and Uchiyama \cite{TU}
proved the following fundamental properties (see also \cite{F}):

\begin{enumerate}
\item[(i)] If $a,$ $b\in A,$ then $a\geq 0,$ $b\geq 0$ imply $a+b\geq 0$ and 
$\alpha \geq 0$ implies $\alpha a\geq 0;$

\item[(ii)] If $a,$ $b\in A,$ then $a>0,$ $b\geq 0$ imply $a+b>0;$

\item[(iii)] If $a,$ $b\in A,$ then either $a\geq b>0$ or $a>b\geq 0$ imply $%
a>0;$

\item[(iv)] If $a>0,$ then $a^{-1}>0;$

\item[(v)] If $c>0,$ then $0<b<a$ if and only if $cbc<cac,$ also $0<b\leq a$
if and only if $cbc\leq cac;$

\item[(vi)] If $0<a<1,$ then $1<a^{-1};$

\item[(vii)] If $0<b<a,$ then $0<a^{-1}<b^{-1},$ also if $0<b\leq a,$ then $%
0<a^{-1}\leq b^{-1}.$
\end{enumerate}

Okayasu \cite{O} showed that the \textit{L\"{o}wner-Heinz inequality}
remains valid in a Hermitian unital Banach $\ast $-algebra with continuous
involution, namely if $a,$ $b\in A$ and $p\in \left[ 0,1\right] $ then $a>b$ 
$\left( a\geq b\right) $ implies that $a^{p}>b^{p}$ $\left( a^{p}\geq
b^{p}\right) .$

In order to introduce the real power of a positive element, we need the
following facts \cite[Theorem 41.5]{BD}.

Let $a\in A$ and $a>0,$ then $0\notin \sigma \left( a\right) $ and the fact
that $\sigma \left( a\right) $ is a compact subset of $\mathbb{C}$ implies
that $\inf \{z:z\in \sigma \left( a\right) \}>0$ and $\sup \{z:z\in \sigma
\left( a\right) \}<\infty .$ Choose $\gamma $ to be close rectifiable curve
in $\{\func{Re}z>0\},$ the right half open plane of the complex plane, such
that $\sigma \left( a\right) \subset \limfunc{ins}\left( \gamma \right) ,$
the inside of $\gamma .$ Let $G$ be an open subset of $\mathbb{C}$ with $%
\sigma \left( a\right) \subset G.$ If $f:G\rightarrow \mathbb{C}$ is
analytic, we define an element $f\left( a\right) $ in $A$ by 
\begin{equation*}
f\left( a\right) :=\frac{1}{2\pi i}\int_{\gamma }f\left( z\right) \left(
z-a\right) ^{-1}dz.
\end{equation*}%
It is well known (see for instance \cite[pp. 201-204]{C}) that $f\left(
a\right) $ does not depend on the choice of $\gamma $ and the Spectral
Mapping Theorem (SMT) 
\begin{equation*}
\sigma \left( f\left( a\right) \right) =f\left( \sigma \left( a\right)
\right)
\end{equation*}
holds.

For any $\alpha \in \mathbb{R}$ we define for $a\in A$ and $a>0,$ the real
power 
\begin{equation*}
a^{\alpha }:=\frac{1}{2\pi i}\int_{\gamma }z^{\alpha }\left( z-a\right)
^{-1}dz,
\end{equation*}%
where $z^{\alpha }$ is the principal $\alpha $-power of $z.$ Since $A$ is a
Banach $\ast $-algebra, then $a^{\alpha }\in A.$ Moreover, since $z^{\alpha
} $ is analytic in $\{\func{Re}z>0\},$ then by (SMT) we have%
\begin{equation*}
\sigma \left( a^{\alpha }\right) =\left( \sigma \left( a\right) \right)
^{\alpha }=\{z^{\alpha }:z\in \sigma \left( a\right) \}\subset \left(
0,\infty \right) .
\end{equation*}

Following \cite{F}, we list below some important properties of real powers:

\begin{enumerate}
\item[(viii)] If $0<a\in A$ and $\alpha \in \mathbb{R}$, then $a^{\alpha
}\in A$ with $a^{\alpha }>0$ and $\left( a^{2}\right) ^{1/2}=a,$ \cite[Lemma
6]{TU};

\item[(ix)] If $0<a\in A$ and $\alpha ,$ $\beta \in \mathbb{R}$, then $%
a^{\alpha }a^{\beta }=a^{\alpha +\beta };$

\item[(x)] If $0<a\in A$ and $\alpha \in \mathbb{R}$, then $\left( a^{\alpha
}\right) ^{-1}=\left( a^{-1}\right) ^{\alpha }=a^{-\alpha };$

\item[(xi)] If $0<a,$ $b\in A$, $\alpha ,$ $\beta \in \mathbb{R}$ and $%
ab=ba, $ then $a^{\alpha }b^{\beta }=b^{\beta }a^{\alpha }.$
\end{enumerate}

We define the following means for $\nu \in \left[ 0,1\right] ,$ see also 
\cite{F} for different notations:%
\begin{equation}
a\nabla _{\nu }b:=\left( 1-\nu \right) a+\nu b,\text{ }a,\text{ }b\in A 
\tag{A}  \label{A}
\end{equation}%
the \textit{weighted arithmetic mean }of $\left( a,b\right) ,$%
\begin{equation}
a!_{\nu }b:=\left( \left( 1-\nu \right) a^{-1}+\nu b^{-1}\right) ^{-1},\text{
}a,\text{ }b>0  \tag{H}  \label{H}
\end{equation}%
the \textit{weighted harmonic mean} of positive elements $\left( a,b\right) $
and 
\begin{equation}
a\sharp _{\nu }b:=a^{1/2}\left( a^{-1/2}ba^{-1/2}\right) ^{\upsilon }a^{1/2}
\tag{G}  \label{G}
\end{equation}%
the \textit{weighted geometric mean }of positive elements $\left( a,b\right)
.$ Our notations above are motivated by the classical notations used in
operator theory. For simplicity, if $\nu =\frac{1}{2},$ we use the simpler
notations $a\nabla b,$ $a!b$ and $a\sharp b.$ The definition of weighted
geometric mean can be extended for any real $\nu .$

In \cite{F}, B. Q. Feng proved the following properties of these means in $A$
a Hermitian unital Banach $\ast $-algebra:

\begin{enumerate}
\item[(xii)] If $0<a,$ $b\in A,$ then $a!b=b!a$ and $a\sharp b=b\sharp a;$

\item[(xiii)] If $0<a,$ $b\in A$ and $c\in \limfunc{Inv}\left( A\right) ,$
then%
\begin{equation*}
c^{\ast }\left( a!b\right) c=\left( c^{\ast }ac\right) !\left( c^{\ast
}bc\right) \text{ and }c^{\ast }\left( a\sharp b\right) c=\left( c^{\ast
}ac\right) \sharp \left( c^{\ast }bc\right) ;
\end{equation*}

\item[(xiv)] If $0<a,$ $b\in A$ and $\nu \in \left[ 0,1\right] $, then 
\begin{equation*}
\left( a!_{\nu }b\right) ^{-1}=\left( a^{-1}\right) \nabla _{\nu }\left(
b^{-1}\right) \text{ and }\left( a^{-1}\right) \sharp _{\nu }\left(
b^{-1}\right) =\left( a\sharp _{\nu }b\right) ^{-1}.
\end{equation*}
\end{enumerate}

Utilising the Spectral Mapping Theorem and the Bernoulli inequality for real
numbers, B. Q. Feng obtained in \cite{F} the following inequality between
the weighted means introduced above: 
\begin{equation}
a\nabla _{\nu }b\geq a\sharp _{\nu }b\geq a!_{\nu }b  \tag{HGA}  \label{HGA}
\end{equation}%
for any $0<a,$ $b\in A$ and $\nu \in \left[ 0,1\right] .$

In \cite{TU}, Tanahashi and Uchiyama obtained the following identity of
interest:

\begin{lemma}
\label{l.1.1}If $0<c,$ $d$ and $\lambda $ is a real number, then%
\begin{equation}
\left( dcd\right) ^{\lambda }=dc^{1/2}\left( c^{1/2}d^{2}c^{1/2}\right)
^{\lambda -1}c^{1/2}d.  \label{e.1.1}
\end{equation}
\end{lemma}

We can prove the following fact:

\begin{proposition}
\label{p.1.1}For any $0<a,$ $b\in A$ we have%
\begin{equation}
b\sharp _{1-\nu }a=a\sharp _{\nu }b  \label{e.1.2}
\end{equation}%
for any real number $\nu .$
\end{proposition}

\begin{proof}
We take in (\ref{e.1.1}) $d=b^{-1/2}$ and $c=a$ to get 
\begin{equation*}
\left( b^{-1/2}ab^{-1/2}\right) ^{\lambda }=b^{-1/2}a^{1/2}\left(
a^{1/2}b^{-1}a^{1/2}\right) ^{\lambda -1}a^{1/2}b^{-1/2}.
\end{equation*}%
If we multiply both sides of this equality by $b^{1/2}$ we get 
\begin{equation}
b^{1/2}\left( b^{-1/2}ab^{-1/2}\right) ^{\lambda }b^{1/2}=a^{1/2}\left(
a^{1/2}b^{-1}a^{1/2}\right) ^{\lambda -1}a^{1/2}.  \label{e.1.3}
\end{equation}%
Since%
\begin{equation*}
\left( a^{1/2}b^{-1}a^{1/2}\right) ^{\lambda -1}=\left[ \left(
a^{1/2}b^{-1}a^{1/2}\right) ^{-1}\right] ^{1-\lambda }=\left(
a^{-1/2}ba^{-1/2}\right) ^{1-\lambda }
\end{equation*}%
then by (\ref{e.1.3}) we get%
\begin{equation*}
a\sharp _{1-\nu }b=b\sharp _{\nu }a.
\end{equation*}%
By swapping in this equality $a$ with $b$ we get the desired result (\ref%
{e.1.2}).
\end{proof}

In this paper we introduce the \textit{quadratic weighted geometric mean}
for invertible elements $x,$ $y$ in a Hermitian unital Banach $\ast $%
-algebra and real number $\nu $. We show that it can be represented in terms
of $\sharp _{\nu },$ which is the usual geometric mean and provide some
inequalities for this mean under various assumptions for the elements
involved.

\section{Quadratic Weighted Geometric Mean}

In what follows we assume that $A$ is a Hermitian unital Banach $\ast $%
-algebra.

We observe that if $x\in \limfunc{Inv}\left( A\right) ,$ then $x^{\ast }\in 
\limfunc{Inv}\left( A\right) ,$ which implies that $x^{\ast }x\in \limfunc{%
Inv}\left( A\right) $. Therefore by Shirali-Ford theorem we have $x^{\ast
}x>0.$ If we define the modulus of the element $c\in A$ by $\left\vert
c\right\vert :=\left( c^{\ast }c\right) ^{1/2}$ then for $c\in \limfunc{Inv}%
\left( A\right) $ we have $\left\vert c\right\vert ^{2}>0$ and by (viii), $%
\left\vert c\right\vert >0.$ If $c>0,$ then by (viii) we have $\left\vert
c\right\vert =c.$

For $x,$ $y\in \limfunc{Inv}\left( A\right) $ we consider the element 
\begin{equation}
d:=\left( x^{\ast }\right) ^{-1}y^{\ast }yx^{-1}=\left( yx^{-1}\right)
^{\ast }yx^{-1}=\left\vert yx^{-1}\right\vert ^{2}.  \label{e.2.1}
\end{equation}%
Since $yx^{-1}\in \limfunc{Inv}\left( A\right) $ then $d>0,$ $d\in \limfunc{%
Inv}\left( A\right) ,$ $d^{-1}=\left\vert yx^{-1}\right\vert ^{-2},$ and
also 
\begin{equation}
d^{-1}=\left( \left( x^{\ast }\right) ^{-1}y^{\ast }yx^{-1}\right)
^{-1}=xy^{-1}\left( y^{-1}\right) ^{\ast }x^{\ast }=\left\vert \left(
y^{-1}\right) ^{\ast }x^{\ast }\right\vert ^{2}.  \label{e.2.2}
\end{equation}%
For $\nu \in \mathbb{R}$, by using the property (viii) we get that $d^{\nu
}=\left\vert yx^{-1}\right\vert ^{2\nu }>0$ and $d^{\nu /2}=\left\vert
yx^{-1}\right\vert ^{\nu }>0$. Since 
\begin{equation*}
x^{\ast }d^{\nu }x=x^{\ast }\left\vert yx^{-1}\right\vert ^{2\nu
}x=\left\vert \left\vert yx^{-1}\right\vert ^{\nu }x\right\vert ^{2}
\end{equation*}%
and $\left\vert yx^{-1}\right\vert ^{\nu }x\in \limfunc{Inv}\left( A\right)
, $ it follows that $x^{\ast }d^{\nu }x>0.$

We introduce the \textit{quadratic weighted mean} of $\left( x,y\right) $
with $x,$ $y\in \limfunc{Inv}\left( A\right) $ and the real weight $\nu \in 
\mathbb{R}$, as the positive element denoted by $x\circledS _{\nu }y$ and
defined by%
\begin{equation}
x\circledS _{\nu }y:=x^{\ast }\left( \left( x^{\ast }\right) ^{-1}y^{\ast
}yx^{-1}\right) ^{\nu }x=x^{\ast }\left\vert yx^{-1}\right\vert ^{2\nu
}x=\left\vert \left\vert yx^{-1}\right\vert ^{\nu }x\right\vert ^{2}. 
\tag{S}  \label{Se}
\end{equation}%
When $\nu =1/2,$ we denote $x\circledS _{1/2}y$ by $x\circledS y$ and we
have 
\begin{equation*}
x\circledS y=x^{\ast }\left( \left( x^{\ast }\right) ^{-1}y^{\ast
}yx^{-1}\right) ^{1/2}x=x^{\ast }\left\vert yx^{-1}\right\vert x=\left\vert
\left\vert yx^{-1}\right\vert ^{1/2}x\right\vert ^{2}.
\end{equation*}

We can also introduce the $1/2$-\textit{quadratic weighted mean} of $\left(
x,y\right) $ with $x,$ $y\in \limfunc{Inv}\left( A\right) $ and the real
weight $\nu \in \mathbb{R}$ by 
\begin{equation}
x\circledS _{\nu }^{1/2}y:=\left( x\circledS _{\nu }y\right)
^{1/2}=\left\vert \left\vert yx^{-1}\right\vert ^{\nu }x\right\vert . 
\tag{$1/2$-S}  \label{Se1/2}
\end{equation}%
Correspondingly, when $\nu =1/2$ we denote $x\circledS ^{1/2}y$ and we have 
\begin{equation*}
x\circledS ^{1/2}y=\left\vert \left\vert yx^{-1}\right\vert
^{1/2}x\right\vert .
\end{equation*}

The following equalities hold:

\begin{proposition}
For any $x,$ $y\in \limfunc{Inv}\left( A\right) $ and $\nu \in \mathbb{R}$
we have%
\begin{equation}
\left( x\circledS _{\nu }y\right) ^{-1}=\left( x^{\ast }\right)
^{-1}\circledS _{\nu }\left( y^{\ast }\right) ^{-1}  \label{e.2.4}
\end{equation}%
and%
\begin{equation}
\left( x^{-1}\right) \circledS _{\nu }\left( y^{-1}\right) =\left( x^{\ast
}\circledS _{\nu }y^{\ast }\right) ^{-1}.  \label{e.2.5}
\end{equation}
\end{proposition}

\begin{proof}
We observe that for any $x,$ $y\in \limfunc{Inv}\left( A\right) $ and $\nu
\in \mathbb{R}$ we have%
\begin{equation*}
\left( x\circledS _{\nu }y\right) ^{-1}=\left( x^{\ast }\left( \left(
x^{\ast }\right) ^{-1}y^{\ast }yx^{-1}\right) ^{\nu }x\right)
^{-1}=x^{-1}\left( xy^{-1}\left( y^{\ast }\right) ^{-1}x^{\ast }\right)
^{\nu }\left( x^{\ast }\right) ^{-1}
\end{equation*}%
and%
\begin{align*}
& \left( x^{\ast }\right) ^{-1}\circledS _{\nu }\left( y^{\ast }\right) ^{-1}
\\
& =\left( \left( x^{\ast }\right) ^{-1}\right) ^{\ast }\left( \left( \left(
\left( x^{\ast }\right) ^{-1}\right) ^{\ast }\right) ^{-1}\left( \left(
y^{\ast }\right) ^{-1}\right) ^{\ast }\left( y^{\ast }\right) ^{-1}\left(
\left( x^{\ast }\right) ^{-1}\right) ^{-1}\right) ^{\nu }\left( x^{\ast
}\right) ^{-1} \\
& =x^{-1}\left( xy^{-1}\left( y^{\ast }\right) ^{-1}x^{\ast }\right) ^{\nu
}\left( x^{\ast }\right) ^{-1},
\end{align*}%
which proves (\ref{e.2.4}).

If we replace in (\ref{e.2.4}) $x$ by $x^{-1}$ and $y$ by $y^{-1}$ we get%
\begin{equation*}
\left( \left( x^{-1}\right) \circledS _{\nu }\left( y^{-1}\right) \right)
^{-1}=x^{\ast }\circledS _{\nu }y^{\ast }
\end{equation*}%
and by taking the inverse in this equality we get (\ref{e.2.5}).
\end{proof}

If we take in (\ref{Se}) $x=a^{1/2}$ and $y=b^{1/2}$ with $a,$ $b>0$ then we
get 
\begin{equation*}
a^{1/2}\circledS _{\nu }b^{1/2}=a\sharp _{\nu }b
\end{equation*}%
for any $\nu \in \mathbb{R}$ that shows that the quadratic weighted mean can
be seen as an extension of the weighted geometric mean\textit{\ }for
positive elements considered in the introduction.

Let $x,$ $y\in \limfunc{Inv}\left( A\right) .$ If we take in the definition
of "$\sharp _{\nu }$" the elements $a=\left\vert x\right\vert ^{2}>0$ and $%
b=\left\vert y\right\vert ^{2}>0$ we also have for real $\nu $ 
\begin{equation*}
\left\vert x\right\vert ^{2}\sharp _{\nu }\left\vert y\right\vert
^{2}=\left\vert x\right\vert \left( \left\vert x\right\vert ^{-1}\left\vert
y\right\vert ^{2}\left\vert x\right\vert ^{-1}\right) ^{\upsilon }\left\vert
x\right\vert =\left\vert x\right\vert \left\vert \left\vert y\right\vert
\left\vert x\right\vert ^{-1}\right\vert ^{2\upsilon }\left\vert
x\right\vert =\left\vert \left\vert \left\vert y\right\vert \left\vert
x\right\vert ^{-1}\right\vert ^{\upsilon }\left\vert x\right\vert
\right\vert ^{2}.
\end{equation*}

It is then natural to ask how the positive elements $x\circledS _{\nu }y$
and $\left\vert x\right\vert ^{2}\sharp _{\nu }\left\vert y\right\vert ^{2}$
do compare, when $x,$ $y\in \limfunc{Inv}\left( A\right) $ and $\nu \in 
\mathbb{R}$ ?

We need the following lemma that provides a slight generalization of Lemma %
\ref{l.1.1}.

\begin{lemma}
\label{l.1.1.a}If $0<c,$ $d\in \limfunc{Inv}\left( A\right) $ and $\lambda $
is a real number, then%
\begin{equation}
\left( dcd^{\ast }\right) ^{\lambda }=dc^{1/2}\left( c^{1/2}\left\vert
d\right\vert ^{2}c^{1/2}\right) ^{\lambda -1}c^{1/2}d^{\ast }.
\label{e.2.5.1}
\end{equation}
\end{lemma}

\begin{proof}
We provide an argument along the lines in the proof of Lemma 7 from \cite{TU}%
.

Consider the functions $F\left( \lambda \right) :=\left( dcd^{\ast }\right)
^{\lambda }$ and $G\left( \lambda \right) :=dc^{1/2}\left( c^{1/2}\left\vert
d\right\vert ^{2}c^{1/2}\right) ^{\lambda -1}c^{1/2}d^{\ast }$ defined for $%
\lambda \in \mathbb{R}$. It is obvious that $F\left( 1\right) =G\left(
1\right) .$

We have 
\begin{align*}
G^{2}\left( \frac{1}{2}\right) & =dc^{1/2}\left( c^{1/2}\left\vert
d\right\vert ^{2}c^{1/2}\right) ^{-1/2}c^{1/2}d^{\ast }dc^{1/2}\left(
c^{1/2}\left\vert d\right\vert ^{2}c^{1/2}\right) ^{-1/2}c^{1/2}d^{\ast } \\
& =dc^{1/2}\left( c^{1/2}\left\vert d\right\vert ^{2}c^{1/2}\right)
^{-1/2}c^{1/2}\left\vert d\right\vert ^{2}c^{1/2}\left( c^{1/2}\left\vert
d\right\vert ^{2}c^{1/2}\right) ^{-1/2}c^{1/2}d^{\ast } \\
& =dcd^{\ast }=F^{2}\left( \frac{1}{2}\right)
\end{align*}%
and 
\begin{align*}
G^{2^{2}}\left( \frac{1}{2^{2}}\right) & =\left( dc^{1/2}\left(
c^{1/2}\left\vert d\right\vert ^{2}c^{1/2}\right) ^{\frac{1-2^{2}}{2^{2}}%
}c^{1/2}d^{\ast }\right) ^{2^{2}} \\
& =dc^{1/2}\left( c^{1/2}\left\vert d\right\vert ^{2}c^{1/2}\right) ^{-\frac{%
3}{4}}c^{1/2}d^{\ast }dc^{1/2}\left( c^{1/2}\left\vert d\right\vert
^{2}c^{1/2}\right) ^{-\frac{3}{4}}c^{1/2}d^{\ast } \\
& dc^{1/2}\left( c^{1/2}\left\vert d\right\vert ^{2}c^{1/2}\right) ^{-\frac{3%
}{4}}c^{1/2}d^{\ast }dc^{1/2}\left( c^{1/2}\left\vert d\right\vert
^{2}c^{1/2}\right) ^{-\frac{3}{4}}c^{1/2}d^{\ast } \\
& =dc^{1/2}\left( c^{1/2}\left\vert d\right\vert ^{2}c^{1/2}\right) ^{-\frac{%
3}{4}}c^{1/2}\left\vert d\right\vert ^{2}c^{1/2}\left( c^{1/2}\left\vert
d\right\vert ^{2}c^{1/2}\right) ^{-\frac{3}{4}}c^{1/2}d^{\ast } \\
& dc^{1/2}\left( c^{1/2}\left\vert d\right\vert ^{2}c^{1/2}\right) ^{-\frac{3%
}{4}}c^{1/2}\left\vert d\right\vert ^{2}c^{1/2}\left( c^{1/2}\left\vert
d\right\vert ^{2}c^{1/2}\right) ^{-\frac{3}{4}}c^{1/2}d^{\ast } \\
& =dc^{1/2}\left( c^{1/2}\left\vert d\right\vert ^{2}c^{1/2}\right) ^{-\frac{%
1}{2}}c^{1/2}d^{\ast }dc^{1/2}\left( c^{1/2}\left\vert d\right\vert
^{2}c^{1/2}\right) ^{-\frac{1}{2}}c^{1/2}d^{\ast } \\
& =dc^{1/2}\left( c^{1/2}\left\vert d\right\vert ^{2}c^{1/2}\right) ^{-\frac{%
1}{2}}c^{1/2}\left\vert d\right\vert ^{2}c^{1/2}\left( c^{1/2}\left\vert
d\right\vert ^{2}c^{1/2}\right) ^{-\frac{1}{2}}c^{1/2}d^{\ast } \\
& =dcd^{\ast }=F^{2^{2}}\left( \frac{1}{2^{2}}\right) .
\end{align*}

By induction we can conclude that $G^{2^{n}}\left( \frac{1}{2^{n}}\right)
=F^{2^{n}}\left( \frac{1}{2^{n}}\right) $ for any natural number $n\geq 0.$
Since for any $a>0$ we have $\left( a^{2}\right) ^{1/2}=a,$ \cite[Lemma 6]%
{TU}, hence $G\left( \frac{1}{2^{n}}\right) =F\left( \frac{1}{2^{n}}\right) $
for any natural number $n\geq 0.$

Since $F\left( \lambda \right) $; $G\left( \lambda \right) $ are analytic on
the real line $\mathbb{R}$ and $\frac{1}{2^{n}}\rightarrow 0$ for $%
n\rightarrow 0$, we deduce that $F\left( \lambda \right) =G\left( \lambda
\right) $ for any $\lambda \in \mathbb{R}$.
\end{proof}

\begin{remark}
\label{r.3.2}The identity (\ref{e.2.5.1}) was proved by. T. Furuta in \cite%
{Fu} for positive operator $c$ and invertible operator $d$ in the Banach
algebra of all bonded linear operators on a Hilbert space by using the polar
decomposition of the invertible operator $dc^{1/2}$.
\end{remark}

\begin{theorem}
\label{t.2.0}If $x,$ $y\in \limfunc{Inv}\left( A\right) $ and $\lambda $ is
a real number, then%
\begin{equation}
x\circledS _{\nu }y=\left\vert x\right\vert ^{2}\sharp _{\nu }\left\vert
y\right\vert ^{2}  \label{e.2.5.2}
\end{equation}
\end{theorem}

\begin{proof}
If we take $d=\left( x^{\ast }\right) ^{-1}$ and $c=\left\vert y\right\vert
^{2}>0$ in (\ref{e.2.5.1}), then we get%
\begin{align*}
\left( \left( x^{\ast }\right) ^{-1}\left\vert y\right\vert
^{2}x^{-1}\right) ^{\lambda }& =\left( x^{\ast }\right) ^{-1}\left\vert
y\right\vert \left( \left\vert y\right\vert \left\vert \left( x^{\ast
}\right) ^{-1}\right\vert ^{2}\left\vert y\right\vert \right) ^{\lambda
-1}\left\vert y\right\vert x^{-1} \\
& =\left( x^{\ast }\right) ^{-1}\left\vert y\right\vert \left( \left\vert
y\right\vert \left( \left( x^{\ast }\right) ^{-1}\right) ^{\ast }\left(
x^{\ast }\right) ^{-1}\left\vert y\right\vert \right) ^{\lambda
-1}\left\vert y\right\vert x^{-1} \\
& =\left( x^{\ast }\right) ^{-1}\left\vert y\right\vert \left( \left\vert
y\right\vert x^{-1}\left( x^{\ast }\right) ^{-1}\left\vert y\right\vert
\right) ^{\lambda -1}\left\vert y\right\vert x^{-1} \\
& =\left( x^{\ast }\right) ^{-1}\left\vert y\right\vert \left( \left\vert
y\right\vert \left( x^{\ast }x\right) ^{-1}\left\vert y\right\vert \right)
^{\lambda -1}\left\vert y\right\vert x^{-1} \\
& =\left( x^{\ast }\right) ^{-1}\left\vert y\right\vert \left( \left\vert
y\right\vert \left\vert x\right\vert ^{-2}\left\vert y\right\vert \right)
^{\lambda -1}\left\vert y\right\vert x^{-1}.
\end{align*}%
If we multiply this equality at left by $x^{\ast }$ and at right by $x$, we
get%
\begin{align*}
x^{\ast }\left( \left( x^{\ast }\right) ^{-1}\left\vert y\right\vert
^{2}x^{-1}\right) ^{\lambda }x& =\left\vert y\right\vert \left( \left\vert
y\right\vert \left\vert x\right\vert ^{-2}\left\vert y\right\vert \right)
^{\lambda -1}\left\vert y\right\vert \\
& =\left\vert y\right\vert \left( \left\vert y\right\vert ^{-1}\left\vert
x\right\vert ^{2}\left\vert y\right\vert ^{-1}\right) ^{1-\lambda
}\left\vert y\right\vert ,
\end{align*}%
which means that 
\begin{equation}
x\circledS _{\nu }y=\left\vert y\right\vert ^{2}\sharp _{1-\nu }\left\vert
x\right\vert ^{2}.  \label{e.2.5.3}
\end{equation}%
By (\ref{e.1.2}) we have for $a=\left\vert x\right\vert ^{2}>0$ and $%
b=\left\vert y\right\vert ^{2}$ that%
\begin{equation}
\left\vert y\right\vert ^{2}\sharp _{1-\nu }\left\vert x\right\vert
^{2}=\left\vert x\right\vert ^{2}\sharp _{\nu }\left\vert y\right\vert ^{2}.
\label{e.2.5.4}
\end{equation}%
Utilising (\ref{e.2.5.3}) and (\ref{e.2.5.4}) we deduce (\ref{e.2.5.2}).
\end{proof}

Now, assume that $f\left( z\right) $ is analytic in the right half open
plane $\{\func{Re}z>0\}$ and for the interval $I\subset \left( 0,\infty
\right) $ assume that $f\left( z\right) \geq 0$ for any $z\in I.$ If $u\in A$
such that $\sigma \left( u\right) \subset I,$ then by (SMT) we have%
\begin{equation*}
\sigma \left( f\left( u\right) \right) =f\left( \sigma \left( u\right)
\right) \subset f\left( I\right) \subset \left[ 0,\infty \right)
\end{equation*}%
meaning that $f\left( u\right) \geq 0$ in the order of $A.$

Therefore, we can state the following fact that will be used to establish
various inequalities in $A.$

\begin{lemma}
\label{l.2.1}Let $f\left( z\right) $ and $g\left( z\right) $ be analytic in
the right half open plane $\{\func{Re}z>0\}$ and for the interval $I\subset
\left( 0,\infty \right) $ assume that $f\left( z\right) \geq g\left(
z\right) $ for any $z\in I.$ Then for any $u\in A$ with $\sigma \left(
u\right) \subset I$ we have $f\left( u\right) \geq g\left( u\right) $ in the
order of $A.$
\end{lemma}

We have the following inequalities between means:

\begin{theorem}
\label{t.2.1}For any $x,$ $y\in \limfunc{Inv}\left( A\right) $ and $\nu \in %
\left[ 0,1\right] $ we have 
\begin{equation}
\left\vert x\right\vert ^{2}\nabla _{\nu }\left\vert y\right\vert ^{2}\geq
x\circledS _{\nu }y\geq \left\vert x\right\vert ^{2}!_{\nu }\left\vert
y\right\vert ^{2}.  \label{e.2.6}
\end{equation}
\end{theorem}

\begin{proof}
1. Follows by the inequality (\ref{HGA}) and representation (\ref{e.2.5.2})

2. A direct proof using Lemma \ref{l.2.1} is as follows.

For $t>0$ and $\nu \in \left[ 0,1\right] $ we have the scalar arithmetic
mean-geometric mean- harmonic mean inequality 
\begin{equation}
1-\nu +\nu t\geq t^{\nu }\geq \left( 1-\nu +\nu t^{-1}\right) ^{-1}.
\label{e.2.6.a}
\end{equation}%
Consider the functions $f\left( z\right) :=1-\nu +\nu z$, $g\left( z\right)
:=z^{\nu }$ and $h\left( z\right) =\left( 1-\nu +\nu z^{-1}\right) ^{-1}$%
where $z^{\nu }$ is the principal of the power function. Then $f(z)$, $%
g\left( z\right) $ and $h\left( z\right) $ are analytic in the right half
open plane $\{\func{Re}z>0\}$ of the complex plane and by (\ref{e.2.6.a}) we
have $f(z)\geq g\left( z\right) \geq h\left( z\right) $ for any $z>0.$

If $0<u\in \limfunc{Inv}\left( A\right) $ and $\nu \in \left[ 0,1\right] ,$
then by Lemma \ref{l.2.1} we get 
\begin{equation*}
1-\nu +\nu u\geq u^{\nu }\geq \left( 1-\nu +\nu u^{-1}\right) ^{-1}.
\end{equation*}

If $x,$ $y\in \limfunc{Inv}\left( A\right) $, then by taking $u=\left\vert
yx^{-1}\right\vert ^{2}\in \limfunc{Inv}\left( A\right) $ we get%
\begin{equation}
1-\nu +\nu \left\vert yx^{-1}\right\vert ^{2}\geq \left\vert
yx^{-1}\right\vert ^{2\nu }\geq \left( 1-\nu +\nu \left\vert
yx^{-1}\right\vert ^{-2}\right) ^{-1}  \label{e.2.7}
\end{equation}%
for any $\nu \in \left[ 0,1\right] .$

If $a>0$ and $c\in \limfunc{Inv}\left( A\right) $ then obviously $c^{\ast
}ac=\left\vert a^{1/2}c\right\vert ^{2}>0.$ This implies that, if $a\geq
b>0, $ then $c^{\ast }ac\geq c^{\ast }bc>0.$

Therefore, if we multiply the inequality (\ref{e.2.7}) at left with $x^{\ast
}$ and at right with $x,$ then we get%
\begin{equation}
x^{\ast }\left( 1-\nu +\nu \left\vert yx^{-1}\right\vert ^{2}\right) x\geq
x^{\ast }\left\vert yx^{-1}\right\vert ^{2\nu }x\geq x^{\ast }\left( 1-\nu
+\nu \left\vert yx^{-1}\right\vert ^{-2}\right) ^{-1}x  \label{e.2.8}
\end{equation}%
for any $\nu \in \left[ 0,1\right] .$

Observe that 
\begin{align*}
x^{\ast }\left( 1-\nu +\nu \left\vert yx^{-1}\right\vert ^{2}\right) x&
=x^{\ast }\left( 1-\nu +\nu \left( x^{\ast }\right) ^{-1}y^{\ast
}yx^{-1}\right) x \\
& =x^{\ast }\left( 1-\nu +\nu \left( x^{\ast }\right) ^{-1}y^{\ast
}yx^{-1}\right) x \\
& =\left( 1-\nu \right) \left\vert x\right\vert ^{2}+\nu \left\vert
y\right\vert ^{2}=\left\vert x\right\vert ^{2}\nabla _{\nu }\left\vert
y\right\vert ^{2}
\end{align*}%
and%
\begin{align*}
x^{\ast }\left( 1-\nu +\nu \left\vert yx^{-1}\right\vert ^{-2}\right)
^{-1}x& =x^{\ast }\left( 1-\nu +\nu \left( \left( x^{\ast }\right)
^{-1}y^{\ast }yx^{-1}\right) ^{-1}\right) ^{-1}x \\
& =x^{\ast }\left( 1-\nu +\nu xy^{-1}\left( y^{\ast }\right) ^{-1}x^{\ast
}\right) ^{-1}x \\
& =x^{\ast }\left( x\left( \left( 1-\nu \right) x^{-1}\left( x^{\ast
}\right) ^{-1}+\nu y^{-1}\left( y^{\ast }\right) ^{-1}\right) x^{\ast
}\right) ^{-1}x \\
& =x^{\ast }\left( x\left( \left( 1-\nu \right) \left( x^{\ast }x\right)
^{-1}+\nu \left( y^{\ast }y\right) ^{-1}\right) x^{\ast }\right) ^{-1}x \\
& =x^{\ast }\left( x^{\ast }\right) ^{-1}\left( \left( 1-\nu \right) \left(
x^{\ast }x\right) ^{-1}+\nu \left( y^{\ast }y\right) ^{-1}\right)
^{-1}x^{-1}x \\
& =\left( \left( 1-\nu \right) \left\vert x\right\vert ^{-2}+\nu \left\vert
y\right\vert ^{-2}\right) ^{-1}=\left\vert x\right\vert ^{2}!_{\nu
}\left\vert y\right\vert ^{2}.
\end{align*}%
Therefore by (\ref{e.2.8}) we get the desired result (\ref{e.2.6}).
\end{proof}

We can define the weighted means for $\nu \in \left[ 0,1\right] $ and the
elements $x,$ $y\in \limfunc{Inv}\left( A\right) $ and $\nu \in \left[ 0,1%
\right] $ by 
\begin{equation*}
x\nabla _{\nu }^{1/2}y:=\left( \left\vert x\right\vert ^{2}\nabla _{\nu
}\left\vert y\right\vert ^{2}\right) ^{1/2}=\left( \left( 1-\nu \right)
\left\vert x\right\vert ^{2}+\nu \left\vert y\right\vert ^{2}\right) ^{1/2}
\end{equation*}%
and 
\begin{equation*}
x!_{\nu }^{1/2}y:=\left( \left\vert x\right\vert ^{2}!_{\nu }\left\vert
y\right\vert ^{2}\right) ^{1/2}=\left( \left( 1-\nu \right) \left\vert
x\right\vert ^{-2}+\nu \left\vert y\right\vert ^{-2}\right) ^{-1/2}.
\end{equation*}

\begin{corollary}
\label{c.2.1}Let $A$ be a Hermitian unital Banach $\ast $-algebra with
continuous involution. Then for any $x,$ $y\in \limfunc{Inv}\left( A\right) $
and $\nu \in \left[ 0,1\right] $ we have 
\begin{equation}
x\nabla _{\nu }^{1/2}y\geq x\circledS _{\nu }^{1/2}y\geq x!_{\nu }^{1/2}y.
\label{e.2.9}
\end{equation}
\end{corollary}

\begin{proof}
It follows by taking the square root in the inequality (\ref{e.2.6} ) and by
using Okayasu's result from the introduction.
\end{proof}

Recall that a $C^{\ast }$\textit{-algebra} $A$ is a Banach $\ast $-algebra
such that the norm satisfies the condition%
\begin{equation*}
\left\Vert a^{\ast }a\right\Vert =\left\Vert a\right\Vert ^{2}\text{ for any 
}a\in A.
\end{equation*}%
If a $C^{\ast }$-algebra $A$ has a unit $1$, then automatically $\left\Vert
1\right\Vert =1.$

It is well know that, if $A$ is a $C^{\ast }$-algebra, then (see for
instance \cite[2.2.5 Theorem]{M}) 
\begin{equation*}
b\geq a\geq 0\text{ implies that }\left\Vert b\right\Vert \geq \left\Vert
a\right\Vert .
\end{equation*}

\begin{corollary}
\label{c.2.2}Let $A$ be a unital $C^{\ast }$-algebra. Then for any $x,$ $%
y\in \limfunc{Inv}\left( A\right) $ and $\nu \in \left[ 0,1\right] $ we have%
\begin{equation}
\left( 1-\nu \right) \left\Vert x\right\Vert ^{2}+\nu \left\Vert
y\right\Vert ^{2}\geq \left\Vert \left( 1-\nu \right) \left\vert
x\right\vert ^{2}+\nu \left\vert y\right\vert ^{2}\right\Vert \geq
\left\Vert \left\vert yx^{-1}\right\vert ^{\nu }x\right\Vert ^{2}.
\label{e.2.9.1}
\end{equation}
\end{corollary}

\section{Refinements and Reverses}

If $X$ is a linear space and $C\subseteq X$ a convex subset in $X$, then for
any convex function $f:C\rightarrow \mathbb{R}$ and any $z_{i}\in
C,r_{i}\geq 0$ for $i\in \left\{ 1,...,k\right\} ,k\geq 2$ with $%
\sum_{i=1}^{k}r_{i}=R_{k}>0$ one has the \textit{weighted Jensen's
inequality:} 
\begin{equation}
\frac{1}{R_{k}}\sum_{i=1}^{k}r_{i}f\left( z_{i}\right) \geq f\left( \frac{1}{%
R_{k}}\sum_{i=1}^{k}r_{i}z_{i}\right) .  \tag{J}  \label{J}
\end{equation}%
If $f:C\rightarrow \mathbb{R}$ is strictly convex and $r_{i}>0$ for $i\in
\left\{ 1,...,k\right\} $ then the equality case hods in (\ref{J}) if and
only if $z_{1}=...=z_{n}.$

By $\mathcal{P}_{n}$ we denote the set of all nonnegative $n$-tuples $\left(
p_{1},...,p_{n}\right) $ with the property that $\sum_{i=1}^{n}p_{i}=1.$
Consider the \textit{normalised Jensen functional}%
\begin{equation*}
\mathcal{J}_{n}\left( f,\mathbf{x,p}\right) =\sum_{i=1}^{n}p_{i}f\left(
x_{i}\right) -f\left( \sum_{i=1}^{n}p_{i}x_{i}\right) \geq 0,
\end{equation*}%
where $f:C\rightarrow \mathbb{R}$ be a convex function on the convex set $C$
and $\mathbf{x}=\left( x_{1},...,x_{n}\right) \in C^{n}$ and $\mathbf{p\in }%
\mathcal{P}_{n}.$

The following result holds \cite{SSDBul}:

\begin{lemma}
\label{l.1}If $\mathbf{p,}$ $\mathbf{q\in }\mathcal{P}_{n}$, $q_{i}>0$ for
each $i\in \left\{ 1,...,n\right\} $ then 
\begin{equation}
\max_{1\leq i\leq n}\left\{ \frac{p_{i}}{q_{i}}\right\} \mathcal{J}%
_{n}\left( f,\mathbf{x,q}\right) \geq \mathcal{J}_{n}\left( f,\mathbf{x,p}%
\right) \geq \min_{1\leq i\leq n}\left\{ \frac{p_{i}}{q_{i}}\right\} 
\mathcal{J}_{n}\left( f,\mathbf{x,q}\right) \left( \geq 0\right) .
\label{e.3.1}
\end{equation}
\end{lemma}

In the case $n=2,$ if we put $p_{1}=1-p,$ $p_{2}=p,$ $q_{1}=1-q$ and $%
q_{2}=q $ with $p\in \left[ 0,1\right] $ and $q\in \left( 0,1\right) $ then
by (\ref{e.3.1}) we get 
\begin{align}
& \max \left\{ \frac{p}{q},\frac{1-p}{1-q}\right\} \left[ \left( 1-q\right)
f\left( x\right) +qf\left( y\right) -f\left( \left( 1-q\right) x+qy\right) %
\right]  \label{e.3.2} \\
& \geq \left( 1-p\right) f\left( x\right) +pf\left( y\right) -f\left( \left(
1-p\right) x+py\right)  \notag \\
& \geq \min \left\{ \frac{p}{q},\frac{1-p}{1-q}\right\} \left[ \left(
1-q\right) f\left( x\right) +qf\left( y\right) -f\left( \left( 1-q\right)
x+qy\right) \right]  \notag
\end{align}%
for any $x,$ $y\in C.$

If we take $q=\frac{1}{2}$ in (\ref{e.3.2}), then we get%
\begin{align}
& 2\max \left\{ t,1-t\right\} \left[ \frac{f\left( x\right) +f\left(
y\right) }{2}-f\left( \frac{x+y}{2}\right) \right]  \label{e.3.3} \\
& \geq \left( 1-t\right) f\left( x\right) +tf\left( y\right) -f\left( \left(
1-t\right) x+ty\right)  \notag \\
& \geq 2\min \left\{ t,1-t\right\} \left[ \frac{f\left( x\right) +f\left(
y\right) }{2}-f\left( \frac{x+y}{2}\right) \right]  \notag
\end{align}%
for any $x,$ $y\in C$ and $t\in \left[ 0,1\right] .$

We consider the scalar weighted arithmetic, geometric and harmonic means
defined by $A_{\nu }\left( a,b\right) :=\left( 1-\nu \right) a+\nu b,$ $%
G_{\nu }\left( a,b\right) :=a^{1-\nu }b^{\nu }$ and $H_{\nu }\left(
a,b\right) =A_{\nu }^{-1}\left( a^{-1},b^{-1}\right) $ where $a,$ $b>0$ and $%
\nu \in \left[ 0,1\right] .$

If we take the convex function $f:\mathbb{R\rightarrow }\left( 0,\infty
\right) $, $f\left( x\right) =\exp \left( \alpha x\right) ,$ with $\alpha
\neq 0,$ then we have from (\ref{e.3.2}) that 
\begin{align}
& \max \left\{ \frac{p}{q},\frac{1-p}{1-q}\right\} \left[ A_{q}\left( \exp
\left( \alpha x\right) ,\exp \left( \alpha y\right) \right) -\exp \left(
\alpha A_{q}\left( a,b\right) \right) \right]  \label{e.3.4} \\
& \geq A_{p}\left( \exp \left( \alpha x\right) ,\exp \left( \alpha y\right)
\right) -\exp \left( \alpha A_{p}\left( a,b\right) \right)  \notag \\
& \geq \min \left\{ \frac{p}{q},\frac{1-p}{1-q}\right\} \left[ A_{q}\left(
\exp \left( \alpha x\right) ,\exp \left( \alpha y\right) \right) -\exp
\left( \alpha A_{q}\left( a,b\right) \right) \right]  \notag
\end{align}%
for any $p\in \left[ 0,1\right] $ and $q\in \left( 0,1\right) $ and any $x,$ 
$y\in \mathbb{R}$.

For $q=\frac{1}{2}$ we have by (\ref{e.3.4}) that 
\begin{align}
& 2\max \left\{ p,1-p\right\} \left[ A\left( \exp \left( \alpha x\right)
,\exp \left( \alpha y\right) \right) -\exp \left( \alpha A\left( a,b\right)
\right) \right]  \label{e.3.5} \\
& \geq A_{p}\left( \exp \left( \alpha x\right) ,\exp \left( \alpha y\right)
\right) -\exp \left( \alpha A_{p}\left( a,b\right) \right)  \notag \\
& \geq 2\min \left\{ p,1-p\right\} \left[ A\left( \exp \left( \alpha
x\right) ,\exp \left( \alpha y\right) \right) -\exp \left( \alpha A\left(
a,b\right) \right) \right]  \notag
\end{align}%
for any $p\in \left[ 0,1\right] $ and any $x,$ $y\in \mathbb{R}$.

If we take $x=\ln a$ and $y=\ln b$ in (\ref{e.3.4}), then we get 
\begin{multline}
\max \left\{ \frac{p}{q},\frac{1-p}{1-q}\right\} \left[ A_{q}\left(
a^{\alpha },b^{\alpha }\right) -G_{q}^{\alpha }\left( a,b\right) \right]
\geq A_{p}\left( a^{\alpha },b^{\alpha }\right) -G_{p}^{\alpha }\left(
a,b\right)  \label{e.3.6} \\
\geq \min \left\{ \frac{p}{q},\frac{1-p}{1-q}\right\} \left[ A_{q}\left(
a^{\alpha },b^{\alpha }\right) -G_{q}^{\alpha }\left( a,b\right) \right]
\end{multline}%
for any $a,$ $b>0,$ for any $p\in \left[ 0,1\right] $, $q\in \left(
0,1\right) $ and $\alpha \neq 0.$

For $q=\frac{1}{2}$ we have by (\ref{e.3.6}) that%
\begin{align}
\max \left\{ p,1-p\right\} \left( b^{\frac{\alpha }{2}}-a^{\frac{\alpha }{2}%
}\right) ^{2}& \geq A_{p}\left( a^{\alpha },b^{\alpha }\right)
-G_{p}^{\alpha }\left( a,b\right)  \label{e.3.7} \\
& \geq \min \left\{ p,1-p\right\} \left( b^{\frac{\alpha }{2}}-a^{\frac{%
\alpha }{2}}\right) ^{2}  \notag
\end{align}%
for any $a,$ $b>0,$ for any $p\in \left[ 0,1\right] $ and $\alpha \neq 0.$

For $\alpha =1$ we get from (\ref{e.3.7}) that%
\begin{align}
\max \left\{ p,1-p\right\} \left( \sqrt{b}-\sqrt{a}\right) ^{2}& \geq
A_{p}\left( a,b\right) -G_{p}\left( a,b\right)  \label{e.3.8} \\
& \geq \min \left\{ p,1-p\right\} \left( \sqrt{b}-\sqrt{a}\right) ^{2} 
\notag
\end{align}%
for any $a,$ $b>0$ and for any $p\in \left[ 0,1\right] ,$ which are the
inequalities obtained by Kittaneh and Manasrah in \cite{KM1} and \cite{KM2}.

For $\alpha =1$ in (\ref{e.3.6}) we obtain 
\begin{multline}
\max \left\{ \frac{p}{q},\frac{1-p}{1-q}\right\} \left[ A_{q}\left(
a,b\right) -G_{q}\left( a,b\right) \right] \geq A_{p}\left( a,b\right)
-G_{p}\left( a,b\right)  \label{e.3.9} \\
\geq \min \left\{ \frac{p}{q},\frac{1-p}{1-q}\right\} \left[ A_{q}\left(
a,b\right) -G_{q}\left( a,b\right) \right] ,
\end{multline}%
for any $a,$ $b>0,$ for any $p\in \left[ 0,1\right] ,$ which is the
inequality (2.1) from \cite{AFK} in the particular case $\lambda =1$ in a
slightly more general form for the weights $p,$ $q.$

We have the following refinement and reverse for the inequality (\ref{e.2.1}%
):

\begin{theorem}
\label{t.3.1}For any $x,$ $y\in \limfunc{Inv}\left( A\right) $ we have for $%
p\in \left[ 0,1\right] $ and $q\in \left( 0,1\right) $ that%
\begin{multline}
\max \left\{ \frac{p}{q},\frac{1-p}{1-q}\right\} \left( \left\vert
x\right\vert ^{2}\nabla _{q}\left\vert y\right\vert ^{2}-x\circledS
_{q}y\right) \geq \left\vert x\right\vert ^{2}\nabla _{p}\left\vert
y\right\vert ^{2}-x\circledS _{p}y  \label{e.3.12} \\
\geq \min \left\{ \frac{p}{q},\frac{1-p}{1-q}\right\} \left( \left\vert
x\right\vert ^{2}\nabla _{q}\left\vert y\right\vert ^{2}-x\circledS
_{q}y\right) .
\end{multline}%
In particular, we have%
\begin{multline}
2\max \left\{ p,1-p\right\} \left( \left\vert x\right\vert ^{2}\nabla
\left\vert y\right\vert ^{2}-x\circledS y\right) \geq \left\vert
x\right\vert ^{2}\nabla _{p}\left\vert y\right\vert ^{2}-x\circledS _{p}y
\label{e.3.12.1} \\
\geq 2\min \left\{ p,1-p\right\} \left( \left\vert x\right\vert ^{2}\nabla
\left\vert y\right\vert ^{2}-x\circledS y\right) ,
\end{multline}%
for any $p\in \left[ 0,1\right] .$
\end{theorem}

\begin{proof}
From the inequality (\ref{e.3.9}) for $a=1$ and $b=t>0$ we have%
\begin{align}
\max \left\{ \frac{p}{q},\frac{1-p}{1-q}\right\} \left( 1-q+qt-t^{q}\right)
& \geq 1-p+pt-t^{p}  \label{e.3.13} \\
& \geq \min \left\{ \frac{p}{q},\frac{1-p}{1-q}\right\} \left(
1-q+qt-t^{q}\right) ,  \notag
\end{align}%
where $p\in \left[ 0,1\right] $ and $q\in \left( 0,1\right) .$

Consider the functions $f\left( z\right) :=\max \left\{ \frac{p}{q},\frac{1-p%
}{1-q}\right\} \left( 1-q+qz-z^{q}\right) $, $g\left( z\right)
:=1-p+pz-z^{p} $ and $h\left( z\right) =\min \left\{ \frac{p}{q},\frac{1-p}{%
1-q}\right\} \left( 1-q+qt-t^{q}\right) $ where $z^{\nu }$, $\nu \in
\{p,q\}, $ is the principal of the power function. Then $f(z)$, $g\left(
z\right) $ and $h\left( z\right) $ are analytic in the right half open plane 
$\{\func{Re}z>0\}$ of the complex plane and and by (\ref{e.3.13}) we have $%
f(z)\geq g\left( z\right) \geq h\left( z\right) $ for any $z>0.$

If $0<u\in \limfunc{Inv}\left( A\right) $ and $\nu \in \left[ 0,1\right] ,$
then by Lemma \ref{l.2.1} we get%
\begin{align}
\max \left\{ \frac{p}{q},\frac{1-p}{1-q}\right\} \left( 1-q+qu-u^{q}\right)
& \geq 1-p+pu-u^{p}  \label{e.3.14} \\
& \geq \min \left\{ \frac{p}{q},\frac{1-p}{1-q}\right\} \left(
1-q+qu-u^{q}\right) ,  \notag
\end{align}%
where $p\in \left[ 0,1\right] $ and $q\in \left( 0,1\right) .$

If $x,$ $y\in \limfunc{Inv}\left( A\right) $, then by taking $u=\left\vert
yx^{-1}\right\vert ^{2}\in \limfunc{Inv}\left( A\right) $ in (\ref{e.3.14})
we have%
\begin{align}
& \max \left\{ \frac{p}{q},\frac{1-p}{1-q}\right\} \left( 1-q+q\left\vert
yx^{-1}\right\vert ^{2}-\left( \left\vert yx^{-1}\right\vert ^{2}\right)
^{q}\right)  \label{e.3.15} \\
& \geq 1-p+p\left\vert yx^{-1}\right\vert ^{2}-\left( \left\vert
yx^{-1}\right\vert ^{2}\right) ^{p}  \notag \\
& \geq \min \left\{ \frac{p}{q},\frac{1-p}{1-q}\right\} \left(
1-q+q\left\vert yx^{-1}\right\vert ^{2}-\left( \left\vert yx^{-1}\right\vert
^{2}\right) ^{q}\right) ,  \notag
\end{align}%
where $p\in \left[ 0,1\right] $ and $q\in \left( 0,1\right) .$

By multiplying the inequality (\ref{e.3.15}) at left with $x^{\ast }$ and at
right with $x$ we get the desired result (\ref{e.3.12}).
\end{proof}

\begin{remark}
\label{r.3.1}If $0<a,$ $b\in A,$ then by taking $x=a^{1/2}$ and $y=b^{1/2}$
in (\ref{e.3.12}) and (\ref{e.3.12.1}) we get%
\begin{align}
\max \left\{ \frac{p}{q},\frac{1-p}{1-q}\right\} \left( a\nabla
_{q}b-a\sharp _{q}b\right) & \geq a\nabla _{p}b-a\sharp _{p}b
\label{e.3.15.1} \\
& \geq \min \left\{ \frac{p}{q},\frac{1-p}{1-q}\right\} \left( a\nabla
_{q}b-a\sharp _{q}b\right) ,  \notag
\end{align}%
for any $p\in \left[ 0,1\right] $ and $q\in \left( 0,1\right) .$

In particular, for $q=1/2$ we have%
\begin{align}
2\max \left\{ p,1-p\right\} \left( a\nabla b-a\sharp b\right) & \geq a\nabla
_{p}b-a\sharp _{p}b  \label{e.3.15.2} \\
& \geq 2\min \left\{ p,1-p\right\} \left( a\nabla b-a\sharp b\right) , 
\notag
\end{align}%
for any $p\in \left[ 0,1\right] .$
\end{remark}

\section{Inequalities Under Boundedness Conditions}

We consider the function $f_{\nu }:[0,\infty )\rightarrow \lbrack 0,\infty )$
defined for $\nu \in \left( 0,1\right) $ by%
\begin{equation*}
f_{\nu }\left( t\right) =1-\nu +\nu t-t^{\nu }=A_{\nu }\left( 1,t\right)
-G_{\nu }\left( 1,t\right) ,
\end{equation*}%
where $A_{\nu }\left( \cdot ,\cdot \right) $ and $G_{\nu }\left( \cdot
,\cdot \right) $ are the scalar arithmetic and geometric means.

The following lemma holds.

\begin{lemma}
\label{l.4.1}For any $t\in \left[ k,K\right] \subset \lbrack 0,\infty )$ we
have%
\begin{equation}
\max_{t\in \left[ k,K\right] }f_{\nu }\left( x\right) =\Delta _{\nu }\left(
k,K\right) :=\left\{ 
\begin{array}{l}
A_{\nu }\left( 1,k\right) -G_{\nu }\left( 1,k\right) \text{ if }K<1, \\ 
\\ 
\max \left\{ A_{\nu }\left( 1,k\right) -G_{\nu }\left( 1,k\right) ,A_{\nu
}\left( 1,K\right) -G_{\nu }\left( 1,K\right) \right\} \text{ } \\ 
\text{if }k\leq 1\leq K, \\ 
\\ 
\text{ }A_{\nu }\left( 1,K\right) -G_{\nu }\left( 1,K\right) \text{ if }1<k%
\end{array}%
\right.  \label{e.4.1}
\end{equation}%
and%
\begin{equation}
\min_{t\in \left[ k,K\right] }f_{\nu }\left( x\right) =\delta _{\nu }\left(
k,K\right) :=\left\{ 
\begin{array}{l}
A_{\nu }\left( 1,K\right) -G_{\nu }\left( 1,K\right) \text{ if }K<1, \\ 
\\ 
0\text{ if }k\leq 1\leq K, \\ 
\\ 
A_{\nu }\left( 1,k\right) -G_{\nu }\left( 1,k\right) \text{ if }1<K.%
\end{array}%
\right.  \label{e.4.2}
\end{equation}
\end{lemma}

\begin{proof}
The function $f_{\nu }$ is differentiable and 
\begin{equation*}
f_{\nu }^{\prime }\left( t\right) =\nu \left( 1-t^{\nu -1}\right) =\nu \frac{%
t^{1-\nu }-1}{t^{1-\nu }},\text{ }t>0,
\end{equation*}%
which shows that the function $f_{\nu }$ is decreasing on $\left[ 0,1\right] 
$ and increasing on $[1,\infty ),$ $f_{\nu }\left( 0\right) =1-\nu ,$ $%
f_{\nu }\left( 1\right) =0,$ $\lim_{t\rightarrow \infty }f_{\nu }\left(
t\right) =\infty $ and the equation $f_{\nu }\left( t\right) =1-\nu $ for $%
t>0$ has the unique solution $t_{\nu }=\nu ^{\frac{1}{\nu -1}}>1.$

Therefore, by considering the $3$ possible situations for the location of
the interval $\left[ k,K\right] $ and the number $1$ we get the desired
bounds (\ref{e.4.1}) and (\ref{e.4.2}).
\end{proof}

\begin{remark}
\label{r.4.1}We have the inequalities 
\begin{equation*}
0\leq f_{\nu }\left( t\right) \leq 1-\nu \text{ for any }t\in \left[ 0,\nu ^{%
\frac{1}{\nu -1}}\right]
\end{equation*}%
and%
\begin{equation*}
1-\nu \leq f_{\nu }\left( t\right) \text{ for any }t\in \left[ \nu ^{\frac{1%
}{\nu -1}},\infty \right) .
\end{equation*}
\end{remark}

Assume that $x,$ $y\in \limfunc{Inv}\left( A\right) $ and the constants $%
M>m>0$ are such that%
\begin{equation}
M\geq \left\vert yx^{-1}\right\vert \geq m.  \label{mM}
\end{equation}%
The inequality (\ref{mM}) is equivalent to 
\begin{equation*}
M^{2}\geq \left\vert yx^{-1}\right\vert ^{2}=\left( x^{\ast }\right)
^{-1}\left\vert y\right\vert ^{2}x^{-1}\geq m^{2}.
\end{equation*}%
If we multiply at left with $x^{\ast }$ and at right with $x$ we get the
equivalent relation%
\begin{equation}
M^{2}\left\vert x\right\vert ^{2}\geq \left\vert y\right\vert ^{2}\geq
m^{2}\left\vert x\right\vert ^{2}.  \label{mM'}
\end{equation}

We have:

\begin{theorem}
\label{t.4.1}Assume that $x,$ $y\in \limfunc{Inv}\left( A\right) $ and the
constants $M>m>0$ are such that either (\ref{mM}), or, equivalently (\ref%
{mM'}) is true. Then we have the inequalities%
\begin{equation}
\Delta _{\nu }\left( m^{2},M^{2}\right) \left\vert x\right\vert ^{2}\geq
\left\vert x\right\vert ^{2}\nabla _{\nu }\left\vert y\right\vert
^{2}-x\circledS _{\nu }y\geq \delta _{\nu }\left( m^{2},M^{2}\right)
\left\vert x\right\vert ^{2},  \label{e.4.3}
\end{equation}%
for any $\nu \in \left[ 0,1\right] ,$ where $\Delta _{\nu }\left( \cdot
,\cdot \right) $ and $\delta _{\nu }\left( \cdot ,\cdot \right) $ are
defined by (\ref{e.4.1}) and (\ref{e.4.2}), respectively.
\end{theorem}

\begin{proof}
From Lemma \ref{l.4.1} we have the double inequality%
\begin{equation*}
\Delta _{\nu }\left( k,K\right) \geq 1-\nu +\nu t-t^{\nu }\geq \delta _{\nu
}\left( k,K\right)
\end{equation*}%
for any $x\in \left[ k,K\right] \subset \left( 0,\infty \right) $ and $\nu
\in \left[ 0,1\right] .$

If $u\in A$ is an element such that $0<k\leq u\leq K,$ then $\sigma \left(
u\right) \subset \left[ k,K\right] $ and by Lemma \ref{l.2.1} we have in the
order of $A$ that%
\begin{equation}
\Delta _{\nu }\left( k,K\right) \geq 1-\nu +\nu u-u^{\nu }\geq \delta _{\nu
}\left( k,K\right)  \label{e.4.4}
\end{equation}%
for any $\nu \in \left[ 0,1\right] .$

If we take $u=\left\vert yx^{-1}\right\vert ^{2},$ then by (\ref{mM}) we
have $0<m^{2}\leq u\leq M^{2}$ and by (\ref{e.4.4}) we get in the order of $%
A $ that%
\begin{equation}
\Delta _{\nu }\left( m^{2},M^{2}\right) \geq 1-\nu +\nu \left\vert
yx^{-1}\right\vert ^{2}-\left\vert yx^{-1}\right\vert ^{2\nu }\geq \delta
_{\nu }\left( m^{2},M^{2}\right)  \label{e.4.5}
\end{equation}%
for any $\nu \in \left[ 0,1\right] .$

If we multiply this inequality at left with $x^{\ast }$ and at right with $x$
we get 
\begin{align}
\Delta _{\nu }\left( m^{2},M^{2}\right) \left\vert x\right\vert ^{2}& \geq
\left( 1-\nu \right) \left\vert x\right\vert ^{2}+\nu x^{\ast }\left\vert
yx^{-1}\right\vert ^{2}x-x^{\ast }\left\vert yx^{-1}\right\vert ^{2\nu }x
\label{e.4.6} \\
& \geq \delta _{\nu }\left( m^{2},M^{2}\right) \left\vert x\right\vert ^{2} 
\notag
\end{align}%
and since $x^{\ast }\left\vert yx^{-1}\right\vert ^{2}x=x^{\ast }\left(
x^{\ast }\right) ^{-1}\left\vert y\right\vert ^{2}x^{-1}x=\left\vert
y\right\vert ^{2}$ and $x^{\ast }\left\vert yx^{-1}\right\vert ^{2\nu
}x=x\circledS _{\nu }y$ we get from (\ref{e.4.6}) the desired result (\ref%
{e.4.3}).
\end{proof}

\begin{corollary}
\label{c.4.1}With the assumptions of Theorem \ref{t.4.1} we have%
\begin{equation}
R\times \left\{ 
\begin{array}{l}
\left( 1-m\right) ^{2}\left\vert x\right\vert ^{2}\text{if }M<1, \\ 
\\ 
\max \left\{ \left( 1-m\right) ^{2},\left( M-1\right) ^{2}\right\}
\left\vert x\right\vert ^{2}\text{ if }m\leq 1\leq M, \\ 
\\ 
\left( M-1\right) ^{2}\left\vert x\right\vert ^{2}\text{ if }1<m,%
\end{array}%
\right.  \label{e.4.7}
\end{equation}%
\begin{equation*}
\geq \left\vert x\right\vert ^{2}\nabla _{\nu }\left\vert y\right\vert
^{2}-x\circledS _{\nu }y\geq r\times \left\{ 
\begin{array}{l}
\left( 1-M\right) ^{2}\left\vert x\right\vert ^{2}\text{ if }M<1, \\ 
\\ 
0\text{ if }m\leq 1\leq M, \\ 
\\ 
\left( m-1\right) ^{2}\left\vert x\right\vert ^{2}\text{ if }1<m,%
\end{array}%
\right. ,
\end{equation*}%
where $\nu \in \lbrack 0,1],$ $r=\min \left\{ 1-\nu ,\nu \right\} $ and $%
R=\max \left\{ 1-\nu ,\nu \right\} .$
\end{corollary}

\begin{proof}
From the inequality (\ref{e.3.8}) we have for $b=t$ and $a=1$ that 
\begin{equation*}
R\left( \sqrt{t}-1\right) ^{2}\geq f_{\nu }\left( t\right) =1-\nu +\nu
t-t^{\nu }\geq r\left( \sqrt{t}-1\right) ^{2}
\end{equation*}%
for any $t\in \lbrack 0,1].$

Then we have 
\begin{equation*}
\Delta _{\nu }\left( m^{2},M^{2}\right) \leq R\times \left\{ 
\begin{array}{l}
\left( 1-m\right) ^{2}\text{ if }M<1, \\ 
\\ 
\max \left\{ \left( 1-m\right) ^{2},\left( M-1\right) ^{2}\right\} \text{ if 
}m\leq 1\leq M, \\ 
\\ 
\left( M-1\right) ^{2}\text{ if }1<m%
\end{array}%
\right.
\end{equation*}%
and%
\begin{equation*}
\delta _{\nu }\left( m^{2},M^{2}\right) \geq r\times \left\{ 
\begin{array}{l}
\left( 1-M\right) ^{2}\text{ if }M<1, \\ 
\\ 
0\text{ if }m\leq 1\leq M, \\ 
\\ 
\left( m-1\right) ^{2}\text{ if }1<m,%
\end{array}%
\right.
\end{equation*}%
which by Theorem \ref{t.4.1} proves the corollary.
\end{proof}

We observe that, with the assumptions of Theorem \ref{t.4.1} and if $A$ is a
unital $C^{\ast }$-algebra, then by taking the norm in (\ref{e.4.3}), we get 
\begin{equation}
\Delta _{\nu }\left( m^{2},M^{2}\right) \left\Vert x\right\Vert ^{2}\geq
\left\Vert \left\vert x\right\vert ^{2}\nabla _{\nu }\left\vert y\right\vert
^{2}-x\circledS _{\nu }y\right\Vert \geq \delta _{\nu }\left(
m^{2},M^{2}\right) \left\Vert x\right\Vert ^{2},  \label{e.4.7.1}
\end{equation}%
for any $\nu \in \left[ 0,1\right] ,$ which, by triangle inequality also
implies that 
\begin{equation}
\Delta _{\nu }\left( m^{2},M^{2}\right) \left\Vert x\right\Vert ^{2}\geq
\left\Vert \left( 1-\nu \right) \left\vert x\right\vert ^{2}+\nu \left\vert
y\right\vert ^{2}\right\Vert -\left\Vert \left\vert yx^{-1}\right\vert ^{\nu
}x\right\Vert ^{2}\geq 0  \label{e.4.7.2}
\end{equation}%
for any $\nu \in \left[ 0,1\right] .$ This provides a reverse for the second
inequality in (\ref{e.2.9.1}).

\begin{remark}
\label{r.4.2}If $0<a,$ $b\in A$ and there exists the constants $0<k<K$ such
that 
\begin{equation}
Ka\geq b\geq ka>0,  \label{kK}
\end{equation}%
then by (\ref{e.4.3}) we get 
\begin{equation}
\Delta _{\nu }\left( k,K\right) a\geq a\nabla _{\nu }b-a\sharp _{\nu }b\geq
\delta _{\nu }\left( k,K\right) a,  \label{e.4.8}
\end{equation}%
while by (\ref{e.4.7}) we get 
\begin{align}
& R\times \left\{ 
\begin{array}{l}
\left( 1-\sqrt{k}\right) ^{2}a\text{ if }K<1, \\ 
\\ 
\max \left\{ \left( 1-\sqrt{k}\right) ^{2},\left( \sqrt{K}-1\right)
^{2}\right\} a\text{ if }m\leq 1\leq M, \\ 
\\ 
\left( \sqrt{K}-1\right) ^{2}a\text{ if }1<k,%
\end{array}%
\right.  \label{e.4.9} \\
&  \notag \\
& \geq a\nabla _{\nu }b-a\sharp _{\nu }b\geq r\times \left\{ 
\begin{array}{l}
\left( 1-\sqrt{K}\right) ^{2}a\text{ if }K<1, \\ 
\\ 
0\text{ if }k\leq 1\leq K, \\ 
\\ 
\left( \sqrt{k}-1\right) ^{2}a\text{ if }1<k%
\end{array}%
\right. ,  \notag
\end{align}%
where $\nu \in \lbrack 0,1],$ $r=\min \left\{ 1-\nu ,\nu \right\} $ and $%
R=\max \left\{ 1-\nu ,\nu \right\} .$
\end{remark}

\end{document}